\documentclass{article}
\usepackage{graphicx} 
\usepackage{amsfonts}
\usepackage{amsmath}
\usepackage{authblk}
\usepackage{url}

\usepackage{amsmath, amsthm, amssymb}
\newtheorem{theorem}{Theorem}[section]
\newtheorem{lemma}[theorem]{Lemma}
\title{A note on geometric colorings of the Moser lattice}
\author[1]{Ákos Dúcz}

\affil[1]{HUN-REN Alfréd Rényi Institute of Mathematics, Budapest, Hungary}

\date{June 2026}

\begin{document}

\maketitle

\begin{abstract}
    In \cite{matolcsi2025fractionalchromaticnumberplane}, Matolcsi et al. show that the fractional chromatic number of the plane is at least 4. Their proof uses a 27-vertex unit-distance graph in the Moser lattice, with geometric fractional chromatic number exactly 4. We show that this bound is tight for graphs in the Moser lattice by exhibiting geometric 4-colorings of the entire lattice. The same colorings also extend to the entire Moser ring.
\end{abstract}
\section{Introduction}
Let $\omega_1 = 1/2 + i\sqrt{3}/2$ and $\omega_3 = 5/6 + i\sqrt{11}/6$.
The Moser lattice is the additive group defined as $L_{Moser} = \{a + b\omega_1 + c\omega_3 + d\omega_1\omega_3 : a,b,c,d\in\mathbb{Z}\}$.

Let $G$ be a finite unit-distance graph, and let $I(G)$ denote the set of all independent sets of $G$. A fractional coloring of $G$ is then a function $\gamma: I(G) \rightarrow \mathbb{R}$, such that

\begin{enumerate}
    \item $\gamma(S) \geq 0$ for all $S \in I(G)$.
    \item $\sum_{x \in S \in I(G)} \gamma(S) \geq 1$ for any vertex $x\in V(G)$.
\end{enumerate}
The fractional chromatic number $\chi_f(G)$ of $G$ is the least possible value of $\sum_{S\in I(G)}\gamma(S)$ achievable by a fractional coloring of $G$. A coloring is non-fractional (or integral) if $\gamma(S) \in \{0,1\}$ for all $S \in I(G)$.

A geometric fractional coloring (first introduced in \cite{density}) satisfies the additional set of constraints

\[3. \sum_{S \subseteq Z \in I(G)} \gamma(Z) = \sum_{S' \subseteq Z' \in I(G)} \gamma(Z') \]

for any two $S, S'$ which are isometric as subsets of the plane. The geometric fractional chromatic number, or GFCN is defined similarly as above, and denoted by $\chi_{gf}(G)$.
In \cite{matolcsi2025fractionalchromaticnumberplane}, a unit-distance graph $G_{27}$ is presented with $\chi_{gf}(G_{27}) = 4$, and the existence of a non-fractional geometric 4-coloring is shown. We extend this coloring to the entire Moser lattice, showing that any graph within this lattice has such a non-fractional geometric 4-coloring.

This explains the observation in \cite{matolcsi2025fractionalchromaticnumberplane} (Appendix A, Figure 2.) that their graph search improved until it attained the maximum value 4, since the entire graph search was confined to the Moser lattice. However, the fact that this value was attained after finitely many steps is not a logical necessity, and remains to be explained.

\section{The structure of integral geometric colorings}

Let $\gamma$ be a non-fractional coloring and index the set $\{S: S\in I(G), \gamma(S) = 1\}$ by [k]. Let $\pi_\gamma: V(G) \rightarrow [k]$ be the function that maps each vertex of $G$ to their respective colors.

For any two vertices $x,y$ of $G$, let $d(x,y)$ denote their Euclidean distance.

\begin{lemma}
    A non-fractional coloring $\gamma$ is geometric if and only if $d(x,y) = d(z,w) \implies (\pi_\gamma(x) = \pi_\gamma(y) \iff \pi_\gamma(z) = \pi_\gamma(w))$ for all $x,y,z,w \in V(G)$. In other words, $\pi_\gamma(x) = \pi_\gamma(y)$ depends only on the distance $d(x,y)$. 
\end{lemma}

\begin{proof}
    Let $\gamma$ be an integral geometric coloring. If there is a distance $d$ and vertices $x,y,z,w$ such that $d(x,y) = d(z,w) = d$, and $\pi_\gamma(x) = \pi_\gamma(y)$ but $\pi_\gamma(z) \not= \pi_\gamma(w)$, then the independent sets $S=\{x, y\}$ and $S'=\{z,w\}$ are clearly isometric, but violate (3). In the other direction, if no such $d$ exists, then it easily follows that any two isometric $S, S' \in I(G)$ are either both monochromatic or both have at least two differently colored vertices. (3) then follows immediately.
\end{proof}

This alternative definition of an integral geometric coloring easily applies to infinite unit-distance graphs, and in particular to lattices (additive subgroups) in the plane.

\begin{lemma}
    Let $L$ be a planar lattice with an integral geometric k-coloring $\gamma$, and let $L_0 = \{ p: p\in L, \pi_\gamma(p) = \pi_\gamma(0) \}$. Then $L_0$ is a sublattice of $L$.
\end{lemma}

\begin{proof}
    For any $p, q \in L_0$, it is immediate that $-p \in L_0$, since $p\in L_0$ depends only on $d(0,p)$. We also have $p-q \in L_0$, since $d(p,q) = d(p-q, 0)$ and $\pi_\gamma(p) = \pi_\gamma(q)$.
\end{proof}

One can also observe that the color classes of $\gamma$ are exactly the cosets of $L_0$.
\newpage
\section{Integral geometric 4-colorings of the Moser lattice}

By a straightforward computation, the graph $G_{27}$ has exactly two non-fractional geometric 4-colorings up to color permutations. These are given by the vertex-colorings

\[(4, 3, 3, 1, 2, 3, 1, 3, 4, 2, 3, 2, 1, 4, 1, 3, 3, 4, 1, 3, 1, 4, 2, 3, 1, 4, 3)\]

and \[(2, 3, 2, 4, 1, 2, 3, 1, 3, 2, 4, 4, 2, 1, 1, 4, 3, 1, 1, 1, 4, 3, 1, 2, 3, 4, 4)\].

Both of these can be generalized to $L_{Moser}$.

\begin{theorem}

    Let $a,b,c,d \in \mathbb{Z}$ denote the coordinates of a vertex $v \in L_{Moser}$. Let $l_1 = a + b + d$ (mod 2) and $l_2 = a + c + d$ (mod 2). Then $\gamma_1$, induced by

    \[\pi_{\gamma_1}(v) = 1 + l_1 + 2l_2\]

    is a geometric 4-coloring of $L_{Moser}$.

\end{theorem}

\begin{proof}
    Let $v$ and $v+z$ be two points in $L_{Moser}$ with $z = A + B\omega_1 + C\omega_3 + D\omega_1\omega_3$. Then $\pi_{\gamma_1}(v) = \pi_{\gamma_1}(v+z)$ if and only if $k_1 = A + B + D \equiv 0$ (mod 2) and $k_2 = A + C + D \equiv 0$ (mod 2).

    Then $|z|^2 = z\overline{z} = \frac{N + M\sqrt{33}}{6}$ where \[
        N = 6(A^2+AB+B^2+C^2+CD+D^2)+10AC+5AD+5BC+10BD
    \]
    and \[
    M = BC-AD.
    \]
    
    A direct calculation mod 2 gives
        $\frac{N+M}{2} \equiv k_1 + k_2 + k_1k_2$ (mod 2). 
        
        We also have
        \[
            k_1\equiv k_2 \equiv0 \,(mod\,2) \iff k_1 + k_2 + k_1k_2 \equiv 0 \,(mod \, 2),
        \]
        
        thus

        \[
            \pi_{\gamma_1}(v) = \pi_{\gamma_1}(v+z) \iff N+M \equiv 0 \, (mod \, 4).
        \]
    But if $|z|^2 = 1$, then $N = 6$ and $M = 0$, so $\pi_{\gamma_1}(v) \not= \pi_{\gamma_1}(v+z)$. Therefore $\gamma_1$ is a proper integral coloring of $L_{Moser}$, and since $\pi_{\gamma_1}(v) = \pi_{\gamma_1}(v+z)$ depends only on $|z|$, it is also geometric by Lemma 2.1.

\end{proof}

The second 4-coloring of $G_{27}$ is generalized by $\gamma_2$ constructed using $l_1 = a+d$ (mod 2) and $l_2 = b + c + d$ (mod 2), and the proof is essentially the same. Interestingly, these are the only two non-fractional geometric 4-colorings of $L_{Moser}$, up to relabeling of the colors. As it turns out, these colorings were already characterized by Tamás Hubai \cite{hubai}, as two of the 22 essentially different 4-colourings of the Moser ring, although their geometricity was not studied.

Both of these also generalize to the entire Moser ring $R_2$, the multiplicative closure of $L_{Moser}$. In fact, using $R_2 = \{\frac{a + b\omega_1 + c\omega_3 + d\omega_1\omega_3}{3^k} : a,b,c,d,k\in\mathbb{Z}\}$, the proof above can be easily adapted by simply ``forgetting'' the exponent $k$. This was noted to the author by Dániel Varga in private discussion.

\begin{theorem}
    $L_{Moser}$ has exactly two integral geometric 4-colorings: $\gamma_1$ and $\gamma_2$, as defined above.
\end{theorem}

\begin{proof}
    Let $\gamma$ be an integral geometric 4-coloring of $L_{Moser}$, and let $L_0$ be defined as in Lemma 2.2. Then the restriction of $\gamma$ to the points in $G_{27}$ must equal one of the two colorings given above.

    Let $v_1, v_2, ..., v_{27}$ be the vertices of $G_{27}$. If we identify $L_{Moser} \cong \mathbb{Z}^4$, we may write these as the columns of the following matrix:
\setcounter{MaxMatrixCols}{30}
\[
{
\setlength{\arraycolsep}{3.6pt}
\begin{pmatrix}
1 & 0 & 2 & 2 & 1 & 2 & 1 & 1 & 1 & 0 & 3 & 3 & 1 & 2 & 2 & 1 & 0 & 0 & 0 & 3 & 2 & 3 & 1 & 2 & 1 & 2 & 3\\
4 & 4 & 3 & 3 & 3 & 3 & 4 & 2 & 3 & 4 & 3 & 2 & 3 & 3 & 2 & 3 & 2 & 3 & 2 & 0 & 1 & 1 & 1 & 1 & 2 & 2 & 1\\
2 & 3 & 0 & 1 & 2 & 2 & 2 & 3 & 3 & 3 & 0 & 1 & 1 & 1 & 2 & 2 & 3 & 3 & 4 & 1 & 1 & 1 & 2 & 2 & 2 & 2 & 0\\
0 & 0 & 1 & 1 & 1 & 1 & 1 & 1 & 1 & 1 & 2 & 2 & 2 & 2 & 2 & 2 & 2 & 2 & 2 & 3 & 3 & 3 & 3 & 3 & 3 & 3 & 4
\end{pmatrix}
}
\]

    In the first coloring, the vertices $\{v_1, v_9, v_{14}, v_{18}, v_{22}\}$ are colored the same. Then the difference vectors
    \[
        h_1 = v_9 - v_1 = (0,-1,1,1),
    \]
    \[
        h_2 = v_{14}-v_1 = (1,-1,-1,2),
    \]
    \[
        h_3 = v_{18}-v_1 = (-1,-1,1,2),
    \]    
    \[
        h_4 = v_{22}-v_1 = (2,-3,-1,3)
    \]
    lie inside $L_0$. Let $H = \langle h_1, h_2, h_3, h_4\rangle$. Since
\[
\det\left(
\begin{array}{rrrr}
 0 & -1 &  1 & 1 \\ 
 1 & -1 & -1 & 2 \\ 
-1 & -1 &  1 & 2 \\ 
 2 & -3 & -1 & 3
\end{array}
\right) = 4.
\]
we have $[L_{Moser} : H] = 4$. But since $H\subseteq L_0$ and $[L_{Moser} : L_0] = 4$, we must have $H = L_0$, and by Lemma 2.2 this forces the coloring to be unique on $L_{Moser}.$

The same argument works for the second coloring using the vertex set
\( \{v_1,v_3,v_6,v_{10},v_{24}\} \).
\end{proof}

\section{Acknowledgements}

The author thanks Dániel Varga for his help and support in writing this paper, and for many helpful discussions. The author was supported by grant NKFIH-153165.

\bibliographystyle{alpha}
\bibliography{ref}

\end{document}